\documentclass[12pt]{amsart}
  
\usepackage{amsmath}
\usepackage{amsthm}
\usepackage{amssymb}
\usepackage{mathrsfs}

\usepackage{booktabs}
\usepackage{siunitx}
\usepackage{multirow}
\usepackage{cite}

\usepackage[left=3cm, right=3cm, bottom=3cm]{geometry}
\usepackage{array}
\usepackage{enumerate}
\usepackage{color}
\usepackage[colorlinks=true, linkcolor=blue, citecolor=blue]{hyperref}
\numberwithin{equation}{section}
\usepackage{enumitem}
\setlist[enumerate,1]{label={\upshape(\roman*)}}

\theoremstyle{theorem}
\newtheorem{theorem}{Theorem}[section]
\newtheorem{cor}[theorem]{Corollary}

\theoremstyle{definition}

\theoremstyle{remark}
\theoremstyle{proof}

\newcommand{\R}{\mathbb{R}}

\newcommand{\I}{\mathrm{I}}

\renewcommand{\H}{\mathrm{H}}

\begin{document}

\title{$W^{1,1}$ stability for the LSI}

\author{Emanuel Indrei \\
\\}
\address{Department of Mathematics and Statistics\\
Sam Houston State University\\
Huntsville, TX \\
USA.}
\begin{abstract}
The logarithmic Sobolev inequality is fundamental in mathematical physics. Associated stability estimates are equivalent to uncertainty principles. Via a second moment bound, $W^{1,1}$ estimates are obtained in one dimension and similar $W_1$-quantitative estimates are investigated.

\end{abstract}
\maketitle

\section{Introduction}
The classical logarithmic Sobolev inequality LSI attributed to L. Gross states
\begin{equation}\label{eq:LSI}
	\delta(f):=\frac{1}{2}\I(f)-\H(f) =\frac{1}{2}\int \frac{|\nabla f|^{2}}{f}d\gamma-\int f\log f d\gamma \ge 0,
\end{equation}
where $d\gamma=(2\pi)^{-\frac{n}{2}}e^{-\frac{|x|^{2}}{2}}dx$, $f$ is normalized, $\sqrt{f}\in W^{1,2}(\R^{n},d\gamma)$, $\&$ $\delta$ is the LSI deficit \cite{3269872e, MR1894593}. 
Carlen \cite{MR1132315} characterized the equality cases: equality holds in \eqref{eq:LSI} if and only if $f_b(x)=e^{b\cdot x-\frac{b^{2}}{2}},$ $b \in \mathbb{R}^n$. 
Note that if $f$ is normalized and centered, equality is valid  if and only if $f=1$. Therefore, a natural problem is to identify a metric $d$, $a>0$, and $\alpha>0$ such that 
\begin{equation} \label{l}
\delta(f) \ge a d^\alpha(f,1),
\end{equation}
with $f$ normalized and centered. In a recent paper, a variant of \eqref{l} with $\alpha=2$, a constant $a>0$, and $d(f,h)=||f-h||_{L^1(\mathbb{R}^n, d\gamma)}$ was established  \cite{Log}. In \cite{e2as49}, $\alpha=2$ was proven to be sharp. Thus, the problem of identifying the strongest norm was a natural consequence which was addressed in \cite{e2as49}.  
Set
$$
f=|u|^2(\frac{x}{\sqrt{2\pi}}).
$$
Then 
$$
\int f e^{-\frac{|x|^{2}}{2}}(2\pi)^{-\frac{n}{2}}dx=     \int |u|^2 e^{-\pi |x|^2}dx.
$$

$$
\frac{1}{2}\int \frac{|\nabla f|^{2}}{f}d\gamma-\int f\ln f d\gamma= \frac{1}{\pi}\int |\nabla u|^2 e^{-\pi |x|^2}dx-\int |u|^2 \ln |u|^2   e^{-\pi |x|^2}dx.
$$
Therefore \eqref{l} has an equivalent version
$$
\pi \delta_*(u):=\int |\nabla u|^2 e^{-\pi |x|^2}dx-\pi \int |u|^2 \ln |u|^2   e^{-\pi |x|^2}dx \ge \kappa  \int \big|u-1\big|^2 e^{-\pi |x|^2}dx
$$
in the space of non-negative functions which satisfy
$$
||u||_{L^2(\mathbb{R}^n, e^{-\pi |x|^2}dx)}=1
$$
$$
\int x |u|^2e^{-\pi |x|^2}dx=0.
$$
Suppose $\delta_*(u_k) \rightarrow 0$ as $k\rightarrow \infty$, 
then \cite{e2as49} yields
$$|u_k|\to 1$$
in $H^1(e^{-\pi |x|^2}dx)$ if and only if

$$
\int |x|^2|u_k(x)|^2e^{-\pi |x|^2}dx \rightarrow \int |x|^2e^{-\pi |x|^2}dx.\\
$$
Thus, some additional assumptions are necessary to prove
\begin{equation} \label{lw}
\delta_*(u) \ge a d^\alpha(u,1)
\end{equation}
with $d(u, g)=||u-g||_{H^1(e^{-\pi |x|^2}dx)}$. Supposing the associated measures satisfy a Poincar\'e inequality via $\lambda>0$, the stability with the sharp $\alpha=2$ and also best constant $a=a(\lambda)>0$ was proven in \cite{MR3567822}.
It has already been underscored in \cite{arXiv:23026} that a uniform upper bound on the second moment is necessary but not sufficient to prove \eqref{lw} with the $H^1=W^{1,2}$ norm if $\delta_*(u) \rightarrow 0$. A sufficient condition was shown with some type of uniform exponential moment condition. Observe a fourth moment bound also yields \eqref{lw} \cite{e2as49}.   The problem of understanding the extra information that is needed with assuming a second moment bound was discussed on \cite[p. 6]{arXiv:23026}. The main result in my paper is that if $n=1$, while  the second moment control does not imply stability in $W^{1,2}$, it indeed implies stability in $W^{1,1}$:

\begin{theorem} \label{p5}
\noindent (1) If $f$ is normalized and centered in $L^1(\mathbb{R}, d\gamma)$ $\&$
$$
m_2(fd\gamma):=\int_{\mathbb{R}}|x|^{2}fd\gamma \le \alpha<\infty,
$$
then there exists $a_\alpha>0$ so that
$$
||f-1||_{W^{1,1}(\mathbb{R}, d\gamma)} \le a_\alpha \Big(\delta^{\frac{1}{4}}(f)+\delta^{\frac{3}{4}}(f)\Big).
$$
\noindent (2) If $u$ is normalized and centered in $L^2(\mathbb{R}, e^{-\pi |x|^2}dx)$ $\&$
$$
\int |x|^2|u(x)|^2e^{-\pi |x|^2}dx \le \alpha<\infty,
$$
then there exists $\overline{a}_\alpha>0$ so that
$$
||u-1||_{W^{1,1}(\mathbb{R}, e^{-\pi |x|^2}dx)} \le  \overline{a}_\alpha(\delta_*(u)^{\frac{1}{4}}+\delta_*(u)).
$$
\end{theorem}

Therefore a natural curiosity is the higher dimensional analog but that remains an open problem. However the following can be proven:
\begin{cor} \label{wl}
\noindent (1) If $f(x_1,x_2,\ldots,x_n)=\Pi_{k=1}^n f_k(x_k)$, $f_k$ is normalized and centered in $L^1(\mathbb{R}, d\gamma)$ $\&$
$$
m_2(f_k) \le \alpha,
$$
then
$$
||f-1||_{W^{1,1}(\mathbb{R}^n, d\gamma)} \le a_{\alpha}n^{3/4} \Big(\delta^{\frac{1}{4}}(f)+\delta^{\frac{3}{4}}(f)\Big),
$$
with $a_{\alpha}$ as in Theorem \ref{p5}.\\

\noindent (2) If $u(x_1,x_2,\ldots,x_n)=\Pi_{k=1}^n u_k(x_k)$, $u_k$ is normalized and centered in $L^2(\mathbb{R}, e^{-\pi |x|^2}dx)$ $\&$
$$
\int |x_k|^2|u_k(x_k)|^2e^{-\pi |x_k|^2}dx_k \le \alpha<\infty,
$$
then 
$$
||u-1||_{W^{1,1}(\mathbb{R}^n, e^{-\pi |x|^2}dx)} \le  n\overline{a}_\alpha(\delta^{\frac{1}{4}}_*(u)+\delta_*(u)),
$$
with $\overline{a}_{\alpha}$ as in Theorem \ref{p5}.\\
\end{cor}

 The \emph{Wasserstein distance} between two probability measures $\mu,\nu$ with $ p\ge 1$, $m_{p}(d\mu)=\int_{\R^{n}}|x|^{p}d\mu<\infty$, $m_{p}(d\nu)=\int_{\R^{n}}|x|^{p}d\nu<\infty$ is 
\begin{align*}
	W_{p}(d\mu,d\nu)=\inf_{\pi} \Big(\iint_{\R^{n}\times \R^{n}} |x-y|^{p}d\pi(x,y)\Big)^{\frac{1}{p}}
\end{align*}
where the infimum is taken over all probability measures $\pi$ on $\R^{n}\times \R^{n}$ with marginals $\mu$ and $\nu$.
In particular, $W_{1}$ is called the Kantorovich--Rubinstein distance. The stability for $W_1$ has already appeared in \cite{IK18}:
	let $fd\gamma$ be a centered probability measure, $m_2(fd\gamma)=\int_{\R^{n}}|x|^{2}fd\gamma \le M<\infty$. There exists a constant $C=C(n,M)>0$ such that 
	\begin{align*}
		\delta(f)
		\geq		C\min\{W_{1}(fd\gamma, d\gamma),W_{1}^{4}(fd\gamma, d\gamma)\}.
	\end{align*}
$W_1$--stability is not true if one merely has finite second moments, therefore  $C=C(n,M)$ cannot be taken independent of $M$ \cite[Theorem 1.2]{MR4455233}.
Examples in \cite{MR4455233} allude to
$$
\delta(f)
		\geq		CW_{1}^{2}(fd\gamma, d\gamma),
$$
with $f \in \{f \ge 0: m_2(fd\gamma)\le M, ||f||_{L^1(\mathbb{R}^n, d\gamma)}=1, \int_{\R^{n}} xfd\gamma=0\}$. Supposing a type of uniform exponential moment condition, this holds:

\begin{theorem} \label{px5}
If $f$ is normalized and centered in $L^1(\mathbb{R}^n, d\gamma)$, $\alpha>0$, $\epsilon>0$, $\&$
$$
\int \int f(x)f(y) e^{\epsilon|x-y|^2}d\gamma(x) d\gamma(y) \le \alpha,
$$
then there exists $a_{\alpha, \epsilon}>0$ so that
$$
W_{1}(fd\gamma,d\gamma) \le a_{\alpha, \epsilon} \delta^{\frac{1}{2}}(u).
$$
Also,  supposing $\epsilon < .25$ combined with
$$
\alpha >\frac{1}{(2\pi)^n}\large(\frac{\pi}{-2\epsilon+.5}\large)^{n},
$$
the best possible $a_{\alpha, \epsilon}$ necessarily has a lower bound:
$$
a_{\alpha, \epsilon} \ge \frac{|nm_1(\gamma)-m_3(\gamma)|}{\sqrt{n}}.
$$
When $n=1$,
$$
a_{\alpha, \epsilon} \ge \sqrt{\frac{2}{\pi}}.
$$
In addition, the exponent on the deficit is sharp.
\end{theorem}

\section{The Proofs}

\begin{proof}[Proof of Theorem \ref{p5}]
The first step is to consider a probability measure $fd\gamma$ $\&$ $T=\nabla \Phi$ the Brenier map which pushes forward $fd\gamma$ to $d\gamma$. The proof of \eqref{eq:LSI} via optimal transport \cite{MR1894593} implies
$$2 \delta(f) \ge \int |T(x)-x+\nabla \ln f|^2 fd\gamma:$$ 
define $\mu:= \Phi-\frac{1}{2}|x|^2$ so that 
$$
\text{det}(I+D^2 \mu(x)) e^{-|x+\nabla \mu(x) |^2/2}=f(x) e^{-|x|^2/2}.
$$
Therefore taking the $\ln$ and integrating,

\begin{align*}
\int f \ln f d\gamma &\le \int f \Big[\Delta \mu - x \cdot \nabla \mu \Big] d\gamma-\frac{1}{2} \int |\nabla \mu|^2 f d\gamma\\
&=-\int \nabla \mu \cdot \nabla f d\gamma-\frac{1}{2} \int |\nabla \mu|^2 f d\gamma\\
&=-\frac{1}{2}\int \Big|\nabla \mu+\frac{\nabla f}{f} \Big|^2 fd\gamma +  \frac{1}{2} \int \frac{|\nabla f|^2}{f}  d\gamma.\\
\end{align*}
Note this readily implies
 $$
 \frac{1}{2}\int \Big|T-x+\nabla \ln f \Big|^2 fd\gamma \le  \frac{1}{2} \int \frac{|\nabla f|^2}{f}  d\gamma-\int f \ln f d\gamma,
 $$
which thanks to Jensen's inequality yields
\begin{equation} \label{j}
 \delta(f) \ge \frac{1}{2}\int |T(x)-x+\nabla \ln f|^2 fd\gamma \ge \frac{1}{2} \Big(\int |T(x)-x+\nabla \ln f| fd\gamma \Big)^2, 
\end{equation}
$$
 \int |\nabla \ln f| fd\gamma -\int |T(x)-x|fd\gamma \le \int |T(x)-x+\nabla \ln f| fd\gamma \le \sqrt{2 \delta(f)}.
$$ 
Hence since $n=1$,
$$
 \int |\nabla \ln f| fd\gamma \le \int |T(x)-x|fd\gamma+\sqrt{2 \delta(f)}=W_1(fd\gamma, d\gamma)+\sqrt{2 \delta(f)}.
$$
Furthermore, the end of the proof of \cite[Proposition C.1]{IK18} in addition to \cite[Corollary 6]{MR3567822} (in the one-dimensional argument, positivity $\&$ the local boundedness can be removed)  imply
\begin{equation} \label{wx}
	\frac{1}{2}\int \frac{|\nabla f|^{2}}{f}d\gamma-\int f\ln f d\gamma
	\geq \frac{1}{4n}\Big(2\int f\ln f d\gamma+(m_{2}(\gamma)-m_{2}(fd\gamma))\Big)^{2},
\end{equation}
\begin{equation} \label{tr}
W_1(fd\gamma, d\gamma) \le a\max\{(H(f)\delta(f))^{\frac{1}{4}}, (H(f)\delta(f))^{\frac{1}{2}}\}.
\end{equation}
Since 
$$
m_{2}(fd\gamma) \le  \alpha,
$$
\begin{align*}
2H&=2\int f\ln f d\gamma \\
&\le 2\delta^{\frac{1}{2}}(f)+\alpha+m_{2}(d\gamma).
\end{align*}
Note
\begin{align*}
&\max\{(H(f)\delta(f))^{\frac{1}{4}}, (H(f)\delta(f))^{\frac{1}{2}}\} \\
&\le t_a\max\{(\delta^{\frac{1}{2}}(f)+r_a)\delta(f))^{\frac{1}{4}}, (\delta^{\frac{1}{2}}(f)+r_a)\delta(f))^{\frac{1}{2}}\}\\
&=t_a\max\{(\delta^{\frac{3}{2}}(f)+r_a\delta(f))^{\frac{1}{4}}, (\delta^{\frac{3}{2}}(f)+r_a\delta(f))^{\frac{1}{2}}\};
\end{align*}
assuming 
$$
l_a \le \delta(f) \le L_a,
$$

$$
W_1(fd\gamma, d\gamma) \le a\max\{(H(f)\delta(f))^{\frac{1}{4}}, (H(f)\delta(f))^{\frac{1}{2}}\} \le c_a =\frac{c_a}{l_a} l_a \le \frac{c_a}{l_a} \delta(f).
$$

Supposing 
$$
\delta(f) \le l_a
$$
and $l_a<<1$, 
$$
\max\{(\delta^{\frac{3}{2}}(f)+r_a\delta(f))^{\frac{1}{4}}, (\delta^{\frac{3}{2}}(f)+r_a\delta(f))^{\frac{1}{2}}\} \le x_a \delta(f)^{\frac{1}{4}}.
$$
Supposing
$$
\delta(f) \ge l_a
$$
$ l_a>>1$ (observe this case is easy since if $\delta(f) \ge l_a$, subject to the second moment assumption, $W_1(fd\gamma, d\gamma) \le t_a$), 
$$
\max\{(\delta^{\frac{3}{2}}(f)+r_a\delta(f))^{\frac{1}{4}}, (\delta^{\frac{3}{2}}(f)+r_a\delta(f))^{\frac{1}{2}}\} \le j_a \delta^{\frac{3}{4}}(f).
$$

Therefore
\begin{align*}
 \int |\nabla f| d\gamma&= \int |\nabla \ln f| fd\gamma \\
 &\le \int |T(x)-x|fd\gamma+\sqrt{2 \delta(f)}\\
 &=W_1(fd\gamma, d\gamma)+\sqrt{2 \delta(f)}\\
 &\le \max\{x_a,j_a\}(\delta^{\frac{1}{4}}(f)+\delta^{\frac{3}{4}}(f))+\sqrt{2 \delta(f)}.
\end{align*}
Thus set
$$
f=|u|^2(\frac{x}{\sqrt{2\pi}}).
$$
Then 
$$
\int f e^{-\frac{|x|^{2}}{2}}(2\pi)^{-\frac{n}{2}}dx=     \int |u|^2 e^{-\pi |x|^2}dx= 1,
$$
$$
\int |\nabla f| d\gamma =\sqrt{\frac{2}{\pi}}\int |\nabla u||u| e^{-\pi |x|^2}dx.
$$

Hence utilizing \cite{Log}
\begin{align*}
&\int |\nabla u|e^{-\pi |x|^2}dx \\
&=\int |\nabla u||u|e^{-\pi |x|^2}dx +\int |\nabla u|(1-|u|)e^{-\pi |x|^2}dx\\
&\le \sqrt{\frac{\pi}{2}}\int |\nabla f| d\gamma+\int |\nabla u||1-u|e^{-\pi |x|^2}dx\\
 &\le \sqrt{\frac{\pi}{2}}\int |\nabla f| d\gamma+(\int |\nabla u|^2e^{-\pi |x|^2}dx)^{1/2}  (\int |1-u|^2 e^{-\pi |x|^2}dx)^{1/2}\\
 & \le m_a(\delta^{\frac{1}{4}}(f)+\delta^{\frac{3}{4}}(f))+(\int |\nabla u|^2e^{-\pi |x|^2}dx)^{1/2}\sqrt{\overline{\kappa}}\Big[  \frac{1}{\pi}\int |\nabla u|^2 e^{-\pi |x|^2}dx-\int |u|^2 \ln |u|^2   e^{-\pi |x|^2}dx \Big]^{1/2}.
\end{align*}

Note
$$
\frac{1}{2}\int \frac{|\nabla f|^{2}}{f}d\gamma-\int f\ln f d\gamma= \frac{1}{\pi}\int |\nabla u|^2 e^{-\pi |x|^2}dx-\int |u|^2 \ln |u|^2   e^{-\pi |x|^2}dx
$$
$$
\int |x|^2|u(x)|^2e^{-\pi |x|^2}dx =\frac{1}{2\pi}\int|x|^2fd\gamma \le \frac{1}{2\pi} \alpha,
$$
which combined with \cite[Theorem 1.17]{IK18} implies
\begin{align*}
&\frac{1}{\pi} \int |\nabla u(x)|^2 e^{-\pi |x|^2}dx \\
&\le \Big|\pi \int |x|^2e^{-\pi |x|^2}dx-\pi \int |x|^2|u(x)|^2e^{-\pi |x|^2}dx\Big|+\sqrt{2n} \Big(\frac{1}{\pi}\int |\nabla u|^2 e^{-\pi |x|^2}dx-\int |u|^2 \ln |u|^2   e^{-\pi |x|^2}dx \Big)^{\frac{1}{2}}\\
&+\Big( \frac{1}{\pi}\int |\nabla u|^2 e^{-\pi |x|^2}dx-\int |u|^2 \ln |u|^2   e^{-\pi |x|^2}dx\Big)\\
& \le q_\alpha(1+\delta^{\frac{1}{2}}_*(u)+\delta_*(u)).
\end{align*}
Therefore utilizing the above estimates,
\begin{align*}
&\int |\nabla u|e^{-\pi |x|^2}dx \\
 &\le \overline{m}_a(\delta^{\frac{1}{4}}_*(u)+\delta^{\frac{3}{4}}_*(u))+(\pi q_\alpha(1+\delta^{\frac{1}{2}}_*(u)+\delta_*(u)))^{1/2}\sqrt{\overline{\kappa}}\Big[  \delta_*(u)\Big]^{1/2}\\
&\le  a_\alpha(\delta^{\frac{1}{4}}_*(u)+\delta_*(u)).
\end{align*}
\end{proof}
\begin{proof}[Proof of  Corollary \ref{wl}]
Observe that 
$$
||f_k-1||_{W^{1,1}(\mathbb{R}, d\gamma)} \le a_\alpha \Big(\delta^{\frac{1}{4}}(f_k)+\delta^{\frac{3}{4}}(f_k)\Big),
$$
$$
\partial_{x_k} f=f_k'(x_k) \Pi_{i \neq k} f_i(x_i),
$$
$$
\delta(f)=\sum_k\delta(f_k),
$$
yield
\begin{align*}
\int |\nabla f| d\gamma&= \int \sqrt{\sum_k (\partial_{x_k} f)^2} d\gamma\\
&\le \int \sum_k |\partial_{x_k} f| d\gamma\\
&= \sum_k \int_{\mathbb{R}^{n-1}} \big(\int_{\mathbb{R}} |f_k'(x_k)| \frac{e^{-\frac{x_k^2}{2}}}{\sqrt{2\pi}} dx_k\big) \Pi_{i \neq k} f_i(x_i) \frac{e^{-\sum_{i\neq k}\frac{x_i^2}{2}}}{(\sqrt{2\pi})^{n-1}} dx_1 dx_2\ldots dx_{k-1} dx_{k+1}\ldots dx_n \\
&=\sum_k \int_{\mathbb{R}} |f_k'(x_k)| \frac{e^{-\frac{x_k^2}{2}}}{\sqrt{2\pi}} dx_k\\
&\le \sum_k a_\alpha \Big(\delta^{\frac{1}{4}}(f_k)+\delta^{\frac{3}{4}}(f_k)\Big)\\
&\le a_\alpha n^{3/4}( \delta^{\frac{1}{4}}(f)+\delta^{\frac{3}{4}}(f)).
\end{align*}
The proof of (2) is similar.
\end{proof}

\begin{proof}[Proof of  Theorem \ref{px5}]
Observe thanks to \eqref{j},
$$
 \delta^{\frac{1}{2}}(f) +\frac{1}{\sqrt{2}}\int |\nabla f| d\gamma \ge  \frac{1}{\sqrt{2}} \int |T(x)-x| fd\gamma \ge  \frac{1}{\sqrt{2}} W_1(fd\gamma, d\gamma); 
$$
in particular, let 
$$
f=|u|^2.
$$
Via \cite[Theorem 1]{arXiv:23026}

\begin{align*}
\int |\nabla f| d\gamma&= 2 \int |u| |\nabla u| d\gamma\\
&\le 2 \sqrt{\int |\nabla u|^2d\gamma}\\
& \le 2(2/\eta(\alpha, \epsilon))^{\frac{1}{2}} \delta^{\frac{1}{2}}(f).
\end{align*}

This yields the conclusion via
$$
a_{\alpha,\epsilon}=2(2/\eta(\alpha, \epsilon))^{1/2} +\sqrt{2}.
$$

Set
$$
f_a(x):=(2a+1)^{\frac{n}{2}}e^{-a|x|^2}.
$$
Several calculations yield
$$
\delta(f_a)=na-\frac{n}{2} \ln(2a+1).
$$
Assuming $\Gamma_1=fd\gamma$ is the first marginal of $\Gamma$, $\Gamma_2=d\gamma$ is the second marginal of $\Gamma$,
 
$$
W_1(fd\gamma, d\gamma)=\inf \int \int |x-y| d\Gamma \ge |\int |x|fd\gamma-\int |y|d\gamma|,
$$

\begin{align*}
 |\int |x|f_ad\gamma-\int |y|d\gamma|&= |\int |x|((2a+1)^{\frac{n}{2}}e^{-a|x|^2}-1)d\gamma|\\
 &=|\int |x|(a(n-|x|^2)+o(a))d\gamma|,\\
\end{align*}

$$
\frac{|\int |x|f_ad\gamma-\int |y|d\gamma|}{a}=|\int |x|(n-|x|^2+\frac{o(a)}{a})d\gamma|,
$$

\begin{align*}
\frac{\delta^{\frac{1}{2}}(f_a)}{W_1(f_ad\gamma, d\gamma)} &\le \frac{\delta^{\frac{1}{2}}(f_a)}{|\int |x|f_ad\gamma-\int |y|d\gamma|} \\
&=\frac{\delta^{\frac{1}{2}}(f_a)}{|\int |x|(a(n-|x|^2)+o(a))d\gamma|} \\
&=\frac{(\frac{\delta(f_a)}{a^2})^{\frac{1}{2}}}{|\int |x|(n-|x|^2+\frac{o(a)}{a})d\gamma|} \\
&=\frac{(\frac{na-\frac{n}{2} \ln(2a+1)}{a^2})^{\frac{1}{2}}}{|\int |x|(n-|x|^2+\frac{o(a)}{a})d\gamma|}, 
\end{align*}

$$
(\frac{na-\frac{n}{2} \ln(2a+1)}{a^2})^{\frac{1}{2}} \rightarrow \sqrt{n},
$$
with $a\rightarrow 0^+$,
$\&$
$$
|\int |x|(n-|x|^2+\frac{o(a)}{a})d\gamma| \rightarrow |nm_1(d\gamma)-m_3(d\gamma)|.
$$
Observe now that if $\epsilon < .25$ $\&$

$$
\alpha >\frac{1}{(2\pi)^n}\large(\frac{\pi}{-2\epsilon+.5}\large)^{n},
$$
one has assuming $a>0$ is small,
$$
W_{1}(f_ad\gamma,d\gamma) \le a_{\alpha, \epsilon} \delta^{\frac{1}{2}}(f_a).
$$
The argument is: observe that 

$$
e^{\epsilon|x-y|^2}\le e^{2\epsilon(|x|^2+ |y|^2)},
$$
which yields

\begin{align*}
\int \int f_a(x)f_a(y) e^{\epsilon|x-y|^2}d\gamma(x) d\gamma(y) &\le (\frac{2a+1}{2\pi})^n\large(\frac{\pi}{a-2\epsilon+.5}\large)^{n};
\end{align*}
supposing
$$
\alpha >\frac{1}{(2\pi)^n}\large(\frac{\pi}{-2\epsilon+.5}\large)^{n},
$$

then when $a>0$ is small, 

$$
(\frac{2a+1}{2\pi})^n\large(\frac{\pi}{a-2\epsilon+.5}\large)^{n}<\alpha
$$

and that implies
$$
W_{1}(f_ad\gamma,d\gamma) \le a_{\alpha, \epsilon} \delta^{\frac{1}{2}}(f_a).
$$
When
$$
W_{1}(fd\gamma,d\gamma) \le a_{\alpha, \epsilon} \delta^{\frac{1}{2}}(f),
$$
and $a_{\alpha, \epsilon}$ is the sharp constant,
\begin{align*}
\frac{1}{a_{\alpha, \epsilon}}&\le \frac{\delta^{\frac{1}{2}}(f_a)}{W_1(f_ad\gamma, d\gamma)}\\
&\le\frac{(\frac{na-\frac{n}{2} \ln(2a+1)}{a^2})^{\frac{1}{2}}}{|\int |x|(n-|x|^2+\frac{o(a)}{a})d\gamma|} \rightarrow \frac{\sqrt{n}}{|nm_1(d\gamma)-m_3(d\gamma)|}.
\end{align*}
When $n=1$,
$$
\frac{\sqrt{n}}{|nm_1(\gamma)-m_3(\gamma)|}=\frac{\sqrt{2\pi}}{2}.
$$
In order to finish the argument, assume via contradiction that there is the estimate
$$
W_{1}(fd\gamma,d\gamma) \le a_{\alpha, \epsilon} \mu(\delta^{\frac{1}{2}}(f)),
$$
$\mu(a)=o(a)$. Therefore
\begin{align*}
W_{1}(f_ad\gamma,d\gamma)& \le a_{\alpha, \epsilon} \mu(\delta^{\frac{1}{2}}(f_a))\\
& =a_{\alpha, \epsilon}\frac{\mu(\delta^{\frac{1}{2}}(f_a))}{\delta^{\frac{1}{2}}(f_a)}\delta^{\frac{1}{2}}(f_a),\\
\end{align*}
and
$$
\limsup_{a\rightarrow 0^+} \frac{\delta^{\frac{1}{2}}(f_a)}{W_{1}(f_ad\gamma,d\gamma)} \le \frac{\sqrt{2\pi}}{2},
$$
easily imply
\begin{align*}
\frac{1}{a_{\alpha, \epsilon}}&\le\limsup_{a\rightarrow 0^+} \frac{\mu(\delta^{\frac{1}{2}}(f_a))}{\delta^{\frac{1}{2}}(f_a)}\frac{\delta^{\frac{1}{2}}(f_a)}{W_{1}(f_ad\gamma,d\gamma)}\\
&=0.
\end{align*}
In particular, this yields the contradiction.
\end{proof}

\bibliography{References}

\providecommand{\bysame}{\leavevmode\hbox to3em{\hrulefill}\thinspace}
\providecommand{\MR}{\relax\ifhmode\unskip\space\fi MR }
\providecommand{\MRhref}[2]{%
  \href{http://www.ams.org/mathscinet-getitem?mr=#1}{#2}
}
\providecommand{\href}[2]{#2}
\begin{thebibliography}{1}

\bibitem{arXiv:23026}
Giovanni Brigati, Jean Dolbeault, and Nikita Simonov, \emph{Stability for the
  logarithmic {S}obolev inequality}, arXiv:2303.12926v2 (2024).

\bibitem{MR1132315}
Eric~A. Carlen, \emph{Superadditivity of {F}isher's information and logarithmic
  {S}obolev inequalities}, J. Funct. Anal. \textbf{101} (1991), no.~1,
  194--211. \MR{1132315}

\bibitem{MR1894593}
Dario Cordero-Erausquin, \emph{Some applications of mass transport to
  {G}aussian-type inequalities}, Arch. Ration. Mech. Anal. \textbf{161} (2002),
  no.~3, 257--269. \MR{1894593}

\bibitem{Log}
Jean Dolbeault, Maria~J. Esteban, Alessio Figalli, Rupert~L. Frank, and Michael
  Loss, \emph{Sharp stability for sobolev and log-sobolev inequalities, with
  optimal dimensional dependence}, arXiv:2209.08651v4 (2023).

\bibitem{MR3567822}
Max Fathi, Emanuel Indrei, and Michel Ledoux, \emph{Quantitative logarithmic
  {S}obolev inequalities and stability estimates}, Discrete Contin. Dyn. Syst.
  \textbf{36} (2016), no.~12, 6835--6853. \MR{3567822}

\bibitem{3269872e}
L.~Gross, \emph{Logarithmic {S}obolev {I}nequalities}, Amer. J. Math
  \textbf{97} (1975), 1061--1083.

\bibitem{e2as49}
Emanuel Indrei, \emph{{S}harp {S}tability for {LSI}}, Mathematics \textbf{11
  (12), 2670} (2023).

\bibitem{IK18}
Emanuel Indrei and Daesung Kim, \emph{Deficit estimates for the logarithmic
  sobolev inequality}, Differential and Integral Equations \textbf{34} (2021),
  no.~7-8, 437--466.

\bibitem{MR4455233}
Daesung Kim, \emph{Instability results for the logarithmic {S}obolev inequality
  and its application to related inequalities}, Discrete Contin. Dyn. Syst.
  \textbf{42} (2022), no.~9, 4297--4320. \MR{4455233}

\end{thebibliography}
\bibliographystyle{amsplain}

\end{document}